\documentclass[12pt]{amsart}

\usepackage{accents}
\usepackage{appendix}  
\usepackage{amsfonts}
\usepackage{amsmath}
\usepackage{amssymb}	
\usepackage{amsthm}
\usepackage{array,booktabs,multirow} 
\usepackage{braket}
\usepackage{cite}
\usepackage{dsfont} 
\usepackage{mathalfa}
\usepackage{mathtools}
\usepackage[shortlabels]{enumitem} 
\usepackage{etoolbox}
\usepackage{float}
\usepackage[hang, flushmargin]{footmisc}
\usepackage{latexsym}
\usepackage{lipsum}
\usepackage{needspace}
\usepackage{tikz}
\usepackage{hyperref}
\usetikzlibrary{matrix,arrows}

\theoremstyle{plain}
\newtheorem{thm}{Theorem}[section]
\newtheorem{cor}[thm]{Corollary}
\newtheorem{lem}[thm]{Lemma}
\newtheorem{prop}[thm]{Proposition}

\theoremstyle{definition}

\newtheorem{defn}[thm]{Definition}

\theoremstyle{remark}

\setlist[enumerate,1]{leftmargin=2em}

\def\N{\mathbb N}
\def\F{\mathbb F}
\def\Z{\mathbb Z}
\def\Ce{\mathfrak{C}}

\title[The Casimir elements of the Racah algebra]{The Casimir elements of the Racah algebra}

\author{Hau-Wen Huang}
\address{
Hau-Wen Huang\\
Department of Mathematics\\
National Central University\\
Chung-Li 32001 Taiwan
}
\email{hauwenh@math.ncu.edu.tw}

\author{Sarah Bockting-Conrad}
\address{
Sarah Bockting-Conrad\\
Department of Mathematical Sciences\\
DePaul University\\
Chicago, Illinois, USA}
\email{sbocktin@depaul.edu}

\thanks{The research of the first author is partially supported by the Ministry of Science and Technology of Taiwan under the project 106-2628-M-008-001-MY4. 
Part of the research was done when the first author was visiting RCPAM, Graduate School of Information Sciences, Tohoku University in summer 2017. He would like to thank Prof. Tanaka and Mr. Watanabe for their hospitality.
}

\begin{document}
\maketitle

\begin{abstract}
Let $\F$ denote a field with ${\rm char\,}\F\not=2$. The Racah algebra $\Re$ is the unital associative $\F$-algebra defined by generators and relations in the following way.  The generators are $A$, $B$, $C$, $D$. The relations assert that 
$$
[A,B]=[B,C]=[C,A]=2D
$$ 
and each of the elements
\begin{gather*}
\alpha=[A,D]+AC-BA,
\qquad
\beta=[B,D]+BA-CB,
\qquad
\gamma=[C,D]+CB-AC
\end{gather*}
is central in $\Re$. Additionally the element $\delta=A+B+C$ is central in $\Re$. 
We call each element in 
\begin{equation*}
D^2+A^2+B^2
+\frac{(\delta+2)\{A,B\}-\{A^2,B\}-\{A,B^2\}}{2}
+A (\beta-\delta)
+B (\delta-\alpha)+\Ce
\end{equation*}
a Casimir element of $\Re$, where $\Ce$ is the commutative subalgebra of $\Re$ generated by $\alpha$, $\beta$, $\gamma$, $\delta$.  The main results of this paper are as follows. 
Each of the following distinct elements is a Casimir element of $\Re$:
\begin{align*}
\Omega_A
=
D^2
+
\frac{B A C
+C A B}{2}
+ A^2
+B \gamma
-C \beta
-A \delta,
\\
\Omega_B
=
D^2
+
\frac{C B A
+A B C}{2}
+ B^2
+C \alpha 
-A \gamma
-B\delta,
\\
\Omega_C
=
D^2
+
\frac{A C B
+B C A}{2}
+ C^2 
+A \beta 
-B\alpha
-C\delta.
\end{align*}
The set $\{\Omega_A,\Omega_B,\Omega_C\}$ is invariant under a faithful $D_6$-action on $\Re$. 
Moreover we show that any Casimir element $\Omega$ is algebraically independent over $\Ce$; if ${\rm char\,}\F=0$ then the center of $\Re$ is $\Ce[\Omega]$.

\bigskip
\noindent
{\bf Keywords:} Racah algebra, Casimir elements, quadratic algebra.
 \hfil\break
\noindent {\bf 2010 Mathematics Subject Classification}. 
Primary: 81R10.  Secondary: 16S37.

\end{abstract}

\section{Introduction}\label{sec:intro}

Physicists have long been interested in the $6j$-symbols, also known as the Racah symbols, because of the important role they play in theoretical physics.  Within the physics community, it is well known that any exactly solvable problem in quantum mechanics can be studied using Lie algebras or groups.

While this approach has been generally fruitful, in some sense, it is too restrictive.  One issue with 
this approach 
is that some of the symmetry of the problem is lost.  As a result,  in the 1980s, physicists began considering the question, ``Could all the exactly solvable problems be stated in the framework of an algebra so as to preserve the idea of the dynamical symmetry method?"\cite{quadratic1991}.  It was in response to this question that Sklyanin introduced the notion of quadratic algebras in his 1982 paper on the lattice model in the theory of magnetism \cite{sklyanin}.  Recall that quadratic algebras -- algebras with quadratic commutators -- are the simplest generalization of Lie algebras which require linear commutators.

A quadratic algebra known as the Racah algebra has been of particular interest.  The Racah algebra  first appeared in the coupling problem of three angular momenta \cite{Levy1965}.  It was in this paper that a $\Z_3$-symmetric presentation for the Racah algebra first appeared.  The algebra was not referred to as the Racah algebra at this time.  Approximately twenty years later, Granovski{\u\i}, Zhedanov, and Lutsenko began working with quadratic algebras with less complicated structures than the ones originally considered by Sklyanin \cite{quadratic1991, zhedanov1988}. 
In the course of this work, Granovski{\u\i} and Zhedanov rediscovered the Racah algebra in an alternate presentation, now referred to as the standard presentation,
and gave a realization for the Racah algebra in terms of the intermediate Casimir operators of the Lie algebra $\mathfrak{su}(2)$  \cite{zhedanov1988}.  
 In \cite{LTRacah}, the Racah algebra was again rediscovered but this time in the context of Leonard triples.  In this paper the authors used the $\Z_3$-symmetric version of the Racah algebra in their classification of Leonard triples of Racah type.  
In addition to the work cited above, connections between the Racah algebra and the Askey scheme of classical orthogonal polynomials, superintegrable models, the Racah problem for the Lie algebra $\mathfrak{su}(1,1)$, and several other problems in physics  
have been explored \cite{Galbert, integrable2014-1,integrable2014-2, Racahproblem2014, R&BI2015,Kalnins1,Kalnins2, Kalnins3,Kalnins4,zhedanov1988, quadratic1991,quadratic1992, quadratic1989, hiddensymmetry1,hiddensymmetry2,hiddensymmetry3}.

In \cite{SH:2017-3}, we study a homomorphism of the Racah algebra into $U(\mathfrak{sl}_2)\otimes  U(\mathfrak{sl}_2)\otimes  U(\mathfrak{sl}_2)$ induced from the Racah problems for $\mathfrak{su}(2)$ and $\mathfrak{su}(1,1)$. 
In studying the representation theory of the Racah algebra, it has also been shown that Leonard pairs naturally arise in connection with its finite-dimensional irreducible representations \cite{LTRacah, lp&awrelation, Alnajjar2010}. A complete classification of the finite-dimensional irreducible representations of the Racah algebra is given in \cite{SH:2019-1} for the case when the underlying field is algebraically closed and has characteristic zero. In \cite{Alnajjar2010, LTRacah}, it was shown how to construct this family of representations of the Racah algebra from the representations of $U(\mathfrak{sl}_2)$. 
Extending this work, we obtain an injection of the Racah algebra into the tensor product of $U(\mathfrak{sl}_2)$ with a polynomial ring in three variables.
Most recently Genest--Vinet--Zhedanov obtained a realization for the Racah algebra from the standard realization for the Bannai--Ito algebra. Furthermore, it is shown that the Racah algebra is isomorphic to a subalgebra of the Bannai--Ito algebra \cite{H:2018-1}.

We now give a brief overview of the present paper.  Let $\F$ denote a field with ${\rm char\,}\F\not=2$. Let $\Re$ denote the unital associative $\F$-algebra  defined by generators and relations in the following way.  The generators are $A$, $B$, $C$, $D$. The relations assert that 
$$
[A,B]=[B,C]=[C,A]=2D
$$ 
and each of the elements
\begin{gather*}
\alpha=[A,D]+AC-BA,
\qquad
\beta=[B,D]+BA-CB,
\qquad
\gamma=[C,D]+CB-AC
\end{gather*}
is central in $\Re$. It follows from the above definition that the element $\delta=A+B+C$ is central in $\Re$. The algebra $\Re$ is a universal analog of the original Racah algebras and is referred to as the {\it Racah algebra} hereafter \cite{Levy1965, R&BI2015}. 
In this paper we give several alternate presentations for $\Re$ and show that $\Re$ admits a faithful $D_6$-action. Let $\Ce$ denote the commutative subalgebra of $\Re$ generated by $\alpha$, $\beta$, $\gamma$, $\delta$.  
Inspired by the work by Genest--Vinet--Zhedanov in \cite[Section 2]{integrable2014-2}, we introduce the following coset of $\Ce$ in $\Re$:
\begin{equation*}
D^2+A^2+B^2
+\frac{(\delta+2)\{A,B\}-\{A^2,B\}-\{A,B^2\}}{2}
+A (\beta-\delta)
+B (\delta-\alpha)+\Ce
\label{eq:abs}
\end{equation*}
We refer to this coset as the \emph{Casimir class} of $\Re$ and call an element of the Casimir class a \emph{Casimir element} of $\Re$.  
We show that each Casimir element is central in $\Re$ and that the Casimir class is invariant under our $D_6$-action.  

The main results of this paper are as follows. 
Each of the following distinct elements is a Casimir element of $\Re$:
\begin{align*}
\Omega_A
=
D^2
+
\frac{B A C
+C A B}{2}
+ A^2
+B \gamma
-C \beta
-A \delta,
\\
\Omega_B
=
D^2
+
\frac{C B A
+A B C}{2}
+ B^2
+C \alpha 
-A \gamma
-B\delta,
\\
\Omega_C
=
D^2
+
\frac{A C B
+B C A}{2}
+ C^2 
+A \beta 
-B\alpha
-C\delta.
\end{align*}
We show that each of $\Omega_A,\Omega_B,\Omega_C$ is algebraically independent over $\Ce$ and that the set $\{\Omega_A,\Omega_B,\Omega_C\}$ is invariant under the faithful $D_6$-action on $\Re$. 
Moreover, we show that every Casimir element 
$\Omega$ is algebraically independent over $\Ce$.  We additionally show that if ${\rm char\,}\F=0$, then the center of $\Re$ is $\Ce[\Omega]$.

The paper is organized as follows.
In Section 2, we establish some conventions used throughout the paper. 
In Section 3, we introduce the Racah algebra $\Re$ and discuss several different presentations for the algebra.
In Section 4, we give a $D_6$-action on $\Re$ and show that it is faithful.  
In Section 5, we give an $\N$-filtration of $\Re$ and use it to develop Poincar\'{e}--Birkhoff--Witt bases for $\Re$ and $\Ce$.  
In Section 6, we introduce a class of $D_6$-invariant elements of $\Re$ called the Casimir class of $\Re$.
In Section 7, we discuss the center of $\Re$ in the case that the ground field has characteristic zero. 

\section{Preliminaries}
Throughout this paper, we adopt the following conventions. Let $\F$ denote a field with ${\rm char}\,\F\neq 2$.  Let $\N$ denote the set of all nonnegative integers.   When we discuss an algebra, we mean a unital associative algebra. When we discuss a subalgebra, we assume that it has the same unit as the parent algebra.
The bracket $[\,,\,]$ stands for the commutator and the curly bracket $\{\,,\,\}$ stands for the anticommutator. 

Let $A$ denote an $\F$-algebra and let $H, K$ denote  $\F$-subspaces of $A$. By $H\cdot K$, we mean the $\F$-subspace of $A$ spanned by $hk$ for all $h\in H$ and $k\in K$. For all $n\in \N$, the notation $H^n$ stands for 
\begin{align*}
 \underbrace{H\cdot H\cdots H.}_\text{$n$ copies}
\end{align*}
For notational convenience, we define $H^0$ to be $\F 1$. 

Following \cite[p.202]{carter}, we define a filtered algebra as follows. 
We call the algebra $A$ an {\it $\N$-filtered algebra} if there exists a sequence of $\F$-subspaces 
$$
A_0\subseteq A_1\subseteq A_2\subseteq\cdots
$$
 of $A$ such that
$\bigcup_i A_i=A$ and $A_i\cdot A_j\subseteq A_{i+j}$ for all $i,j\in\N$.
In this case, we refer to the sequence $\{A_i\}_{i\in\N}$ as an $\N$-{\it filtration} of $A$.

\section{The Racah algebra $\Re$ and its presentations}

In this section, we introduce the Racah algebra $\Re$ and give three presentations of $\Re$ which will be of use  later in the paper.  

\begin{defn}\label{def:Re}
(\!\!\cite[Equations (2.8) \& (2.9)]{Levy1965}). 
Define an $\F$-algebra $\Re$ by generators and relations in the following way. The generators are $A$, $B$, $C$, $D$. The relations assert that
\begin{gather}\label{r:D}
[A,B]=[B,C]=[C,A]=2D
\end{gather}
and each of the elements 
\begin{gather*}
[A,D]+AC-BA,
\qquad
[B,D]+BA-CB,
\qquad
[C,D]+CB-AC
\end{gather*}
is central in $\Re$.  The $\F$-algebra $\Re$ is called the {\it Racah algebra}. 
\end{defn}

For notational convenience, we let 
\begin{eqnarray}
\alpha &=& [A,D]+AC-BA, \label{r:alpha} \\
\beta &=& [B,D]+BA-CB,\label{r:beta}\\
\gamma &=& [C,D]+CB-AC,\label{r:gamma}\\
\delta &=& A+B+C. \label{r:delta}
\end{eqnarray}
We now make a few comments regarding Definition \ref{def:Re}. The above presentation for $\Re$ is $\Z_3$-symmetric with respect to the generators $A$, $B$, $C$.
We also note from the following lemma that the collections of generators and relations given in Definition \ref{def:Re} are not minimal.

\begin{lem}\label{lem:alp&bet&gam&del}
\begin{enumerate}
\item $\Re$ is generated by $A$, $B$, $C$. 

\item Each of  $\alpha$, $\beta$, $\gamma$, $\delta$ is central in $\Re$.

\item The sum of $\alpha$, $\beta$, $\gamma$ is equal to zero.
\end{enumerate}
\end{lem}
\begin{proof}
By (\ref{r:D}), $D$ can be expressed in terms of $A,B$.  Assertion (i) follows from this along with Definition \ref{def:Re}.

We now show (ii).  
By Definition \ref{def:Re}, the elements $\alpha$, $\beta$, $\gamma$ are central in $\Re$. 
By (\ref{r:D}), the element $\delta$ commutes with $A$, $B$, $C$. It follows from this fact along with (i) that $\delta$ is central in  $\Re$.  So (ii) holds.

We now show (iii).
Using (\ref{r:alpha})--(\ref{r:gamma}) to evaluate the sum $\alpha+\beta+\gamma$, we find that
$$
\alpha+\beta+\gamma=[\delta,D]. 
$$
By (ii),  $[\delta,D]=0$.  So (iii) holds.
\end{proof}

We now give the first of our three alternate presentations for $\Re$.

\begin{prop}\label{prop:presentation1}
The $\F$-algebra $\Re$ has a presentation given by generators $A$, $B$, $C$, $D$, $\alpha$, $\beta$ and relations
\begin{align}
BA &= AB-2D,\label{eq:pres1-1}\\
CB &= BC-2D,\\
CA &= AC+2D,\\
DA &= AD-BA+AC-\alpha,\\
DB &= BD-CB+BA-\beta,\\
\alpha A &= A\alpha,
\qquad
\alpha B = B\alpha,
\qquad
\alpha C = C\alpha,
\\
\beta A &= A\beta,
\qquad
\beta B = B\beta,
\qquad
\beta C = C\beta.
\end{align}
\end{prop}
\begin{proof}
By Lemma \ref{lem:alp&bet&gam&del}(iii), $\gamma=-\alpha-\beta$.  Use this fact to eliminate $\gamma$ from the presentation of $\Re$ given in Definition \ref{def:Re}.  
Reformulate the resulting presentation to obtain the desired result.
\end{proof}

\begin{prop}\label{prop:presentation2}
The $\F$-algebra $\Re$ has a presentation given by generators $A$, $D$, $B$, $\alpha$, $\beta$, $\delta$ and relations
\begin{align}
BA&=AB-2D, \label{eq:pres2-1}
\\
DA&=AD+A\delta-A^2-2AB+2D-\alpha,\label{eq:pres2-2}
\\
BD&=DB+ B\delta-B^2-2AB+2D+\beta,\label{eq:pres2-3}
\\
\alpha A &= A\alpha, \qquad 
\beta A = A\beta, \qquad
\delta A = A\delta, 
\\
\alpha B &= B\alpha, \ \ \ \ \
\beta B = B\beta, \qquad
\delta B = B\delta, 
\\
 \alpha \delta &= \delta \alpha, 
\qquad \
 \beta \delta = \delta \beta. 
\end{align}
\end{prop}
\begin{proof}
By \eqref{r:delta}, $C=\delta-A-B$.  Use this to eliminate $C$ from the presentation of $\Re$ given in Proposition \ref{prop:presentation1}.  Use \eqref{eq:pres1-1} to simplify the resulting relations in order to obtain the desired result.
\end{proof}

\begin{lem}\label{lem:alternate}
The following equations hold in $\Re$:
\begin{align}
D&=\frac{[A,B]}{2},
\label{r:D'}
\\
\alpha&=
\frac{[A,[A,B]]}{2}
+ A \delta-A^2
-\{A,B\},
\label{r:alpha'}
\\
\beta&=
\frac{[B,[A,B]]}{2}
-B \delta 
+B^2
+\{A,B\}.
\label{r:beta'}
\end{align}
\end{lem}
\begin{proof}
Equation (\ref{r:D'}) can be obtained from  (\ref{eq:pres2-1}). 
Equation (\ref{r:alpha'}) (resp. (\ref{r:beta'})) can be obtained from \eqref{eq:pres2-2} (resp. \eqref{eq:pres2-3}) by using \eqref{r:D'} to eliminate $D$.
\end{proof}

\begin{prop}\label{prop:presentation3}
The $\F$-algebra $\Re$ has a presentation given by generators $A$, $B$, $\delta$ and relations
\begin{gather*}
[A,[A,[A,B]]]=2[A^2,B],
\\
[B,[B,[B,A]]]=2[B^2,A],
\\
[B,[A,[B,A]]]=2[A,B^2]-2[B,A^2]-2[A,B]\delta,
\\
[A,[B,[A,B]]]=2[B,A^2]-2[A,B^2]-2[B,A]\delta,
\\
\delta A=A\delta,
\qquad 
\delta B=B\delta.
\end{gather*}
\end{prop}
\begin{proof}
To obtain the result, use Lemma \ref{lem:alternate} to express the relations in Proposition \ref{prop:presentation2} in terms of $A$, $B$, $\delta$. 
\end{proof}

\section{A faithful $D_6$-action on $\Re$}\label{section:D6}

Recall the dihedral group $D_6$  with presentation given by generators $\sigma$, $\tau$ and relations 
\begin{gather*}
\sigma^2=1,
\qquad 
\tau^6=1,
\qquad 
(\sigma\tau)^2=1.
\end{gather*}
In this section, we give an action $D_6$ on the Racah algebra $\Re$ and show that this action is faithful.

\begin{prop}\label{prop:D6}
There exists a unique $D_6$-action on $\Re$ such that each of the following holds:
\begin{enumerate}
\item $\sigma$ acts on $\Re$ as an $\F$-algebra antiautomorphism in the following way:

\begin{table}[H]
\centering
\extrarowheight=3pt
\begin{tabular}{c|cccc|cccc}
$u$  &$A$ &$B$ &$C$ &$D$
&$\alpha$ &$\beta$ &$\gamma$ &$\delta$
\\

\midrule[1pt]

$\sigma(u)$ &$B$ &$A$ &$C$ &$D$
&$-\beta$ &$-\alpha$ &$-\gamma$ &$\delta$
\end{tabular}
\end{table}

\item $\tau$ acts on $\Re$ as an $\F$-algebra antiautomorphism  in the following way:

\begin{table}[H]
\centering
\extrarowheight=3pt
\begin{tabular}{c|cccc|cccc}
$u$  &$A$ &$B$ &$C$ &$D$
&$\alpha$ &$\beta$ &$\gamma$ &$\delta$
\\

\midrule[1pt]

$\tau(u)$ &$B$ &$C$ &$A$ &$-D$
&$\beta$ &$\gamma$ &$\alpha$ &$\delta$
\end{tabular}
\end{table}
\end{enumerate}
\end{prop}
\begin{proof}
By Proposition \ref{prop:presentation2}, there exists a unique $\F$-algebra antiautomorphism $\mathcal S$ of $\Re$ that sends 
\begin{gather*}
A\mapsto B,\qquad 
B\mapsto A,\qquad 
D\mapsto D, \qquad 
\alpha \mapsto -\beta, \qquad 
\beta \mapsto -\alpha, \qquad 
\delta\mapsto \delta.
\end{gather*}
Observe that $\mathcal S$ is of order $2$. It follows from (\ref{r:delta}) that the element $C$ is invariant under $\mathcal S$.  
 By Lemma \ref{lem:alp&bet&gam&del}(iii), the image of $\gamma$ under $\mathcal S$ is $-\gamma$. 

By the comments below Definition \ref{def:Re}, there exists a unique $\F$-algebra antiautomorphism $\mathcal T$ of $\Re$ that sends
\begin{gather*}
A \mapsto B, \qquad 
B \mapsto C, \qquad 
C \mapsto A.
\end{gather*}
It follows from (\ref{r:D})--(\ref{r:gamma}) that $\mathcal T$ sends $D$ to $-D$ and $\alpha$, $\beta$, $\gamma$ to $\beta$, $\gamma$, $\alpha$ respectively. 
It follows from (\ref{r:delta}) that the element $\delta$ is invariant under $\mathcal T$. Observe that  $\mathcal T$ is of order $6$. 

Therefore the composition $\mathcal S\circ \mathcal T$ is an $\F$-algebra automorphism of $\Re$ which sends 
\begin{gather*}
A \mapsto A, \qquad 
B \mapsto C, \qquad 
C \mapsto B, \qquad 
D \mapsto -D.
\end{gather*}
Thus the order of $\mathcal S\circ \mathcal T$  is $2$. 
We have shown that there exists a unique $D_6$-action on $\Re$ such that $\sigma$ and $\tau$ act as $\mathcal S$ and $\mathcal T$, respectively. The proposition follows.
\end{proof}

\begin{lem}\label{lem:faithful}
Let $N$ denote the cyclic subgroup of $D_6$ generated by $\tau$. If $K$ is a normal subgroup of $D_6$ with $K\cap N=\{1\}$, then $K=\{1\}$.
\end{lem}
\begin{proof}
Observe that $\{1, \sigma\}$ is a system of left coset representatives for $N$ in $D_6$. Thus it suffices to show that 
$K\cap \sigma N=\emptyset$.
Suppose $K\cap \sigma N\neq\emptyset$.  Then there exists an integer $0\leq i\leq 5$ such that $\sigma\tau^i\in K$. 
By the normality of $K$, we have $
\tau^{i}\sigma=\sigma^{-1}\cdot \sigma\tau^i \cdot \sigma\in K$. 
Hence $
\tau^{2i}=\tau^{i}\sigma\cdot \sigma\tau^i \in K$. 
Since $K\cap N=\{1\}$ and $\tau$ is of order $6$, it follows that $i=0$ or $i=3$. This shows that 
\begin{gather}\label{e:rhotau3}
K\cap \sigma N\subseteq\{\sigma,\sigma\tau^3\}.
\end{gather}
Suppose that $\sigma\in K\cap \sigma N$. 
By the normality of $K$, 
$$
\sigma\tau^4=\tau\cdot \sigma\cdot \tau^{-1}\in K,
$$ 
which contradicts (\ref{e:rhotau3}). Hence $\sigma\not\in K\cap\sigma N$. By a similar argument, we have $\sigma\tau^3\not\in K\cap \sigma N$. So we see that 
$K\cap \sigma N=\emptyset$
and the result follows.
\end{proof}

\begin{prop}
With reference to Proposition \ref{prop:D6},  
the $D_6$-action on $\Re$ is faithful.
\end{prop}
\begin{proof}
Recall from the proof of Proposition \ref{prop:D6} that the antiautomorphism $\mathcal T:\Re\to \Re$ is of order $6$.  The result follows from this and  Lemma \ref{lem:faithful}.
\end{proof}

\section{An $\N$-filtration of $\Re$}

In this section, we give a Poincar\'{e}--Birkhoff--Witt basis for $\Re$.  Using this basis, we obtain 
an $\N$-filtration for $\Re$ and thus show that $\Re$ is a filtered algebra. 
We begin with the following Poincar\'{e}--Birkhoff--Witt basis for $\Re$:

\begin{thm}\label{thm:basisURA}
The elements
\begin{gather}
A^i D^j B^k \alpha^r \delta^s \beta^t 
\qquad \quad
\hbox{for all $i,j,k,r,s,t\in \N$}\label{eq:basisURA}
\end{gather}
form an $\F$-basis of $\Re$.
\end{thm}
\begin{proof}
To show this we invoke the Diamond Lemma \cite[Theorem 1.2]{bergman}.
Let $M$ denote the free monoid on the alphabet set $S=\{A,B,D,\alpha,\beta,\delta\}$. 
Let $\ell:M\to \N$ denote the length function of $M$. 
Given an element $w=s_1s_2\cdots s_n\in M$ with $s_1,s_2,\ldots,s_n\in S$, 
we call an operation on $w$ an {\it elementary operation} if it consists of performing one of the following actions on $w$:
\begin{enumerate}
\item[$\bullet$]  
Interchange $s_i$ and $s_j$, where $1\leq i<j\leq n$ and the position of $s_j$ is to the left of the position of $s_i$ in the list
$$
A,\quad D,\quad B,\quad 
\alpha, \quad \delta, \quad \beta.
$$

\item[$\bullet$] 
Choose an integer $i$ such that $1\leq i\leq n$ and $s_i\in \{A,B,D\}$, then replace $s_i$ by the left neighbor of $s_i$ in the list 
$$
\delta,
\quad
A,
\quad
B,
\quad
D.
$$
\end{enumerate}
We define a binary relation $\preceq$ on $M$ as follows. For any $u,w\in M$, we say that 
$u\rightarrow w$ whenever $\ell(u)<\ell(w)$ or $u$ is obtained from $w$ by an elementary operation. 
We say that $u\preceq w$ if there exist $u_0,u_1,\ldots,u_k\in M$, with $k\in\N$, 
 such that 
\begin{gather*}
u=u_0\rightarrow u_1\rightarrow \cdots \rightarrow u_{k-1}\rightarrow u_k=w.
\end{gather*}
By construction, $\preceq$ is a monoid partial order on $M$  and satisfies the descending chain condition \cite[p.181]{bergman}.  
By Lemma \ref{lem:alp&bet&gam&del}(ii), we may extend the relations given in Proposition \ref{prop:presentation2}  as follows:
\begin{align*} 
BA &= AB-2D,
\\
DA &= AD+A\delta-A^2-2AB+2D-\alpha,
\\
BD &= DB+B\delta-B^2-2AB+2D+\beta,
\\
\alpha A&=A\alpha, 
\qquad 
\beta A=A\beta, 
\qquad 
\delta A=A\delta,
\\
\alpha B&=B\alpha,
\qquad 
\beta B=B\beta,
\qquad 
\delta B=B\delta,
\\
\alpha D&=D\alpha,
\qquad 
\beta D=D\beta,
\qquad 
\delta D=D\delta,
\\
\beta\alpha &=\alpha\beta,
\qquad 
\delta\alpha=\alpha\delta,
\qquad 
\beta\delta=\delta\beta.
\end{align*}
We regard the above relations as a reduction system. The irreducible elements under the reduction system are exactly the $\F$-linear combinations of the words in (\ref{eq:basisURA}). 
By construction, $\preceq$ is compatible with the reduction system \cite[p.181]{bergman}. 

We now resolve any ambiguities arising from the reduction system.  There are two potential kinds of ambiguities: inclusion ambiguities and overlap ambiguities \cite[p.181]{bergman}. There are no inclusion ambiguities in the above reduction system.  There is one nontrivial overlap ambiguity.  Observe that the word $BDA$ can be reduced in two different ways.
We could eliminate $BD$ first or eliminate $DA$ first.  Either way, one finds that
\begin{align*}
BDA&=ADB-2D^2-3A^2B-3AB^2+6AD+6DB+2AB\delta-2D\delta\\
&\qquad-2A^2-2B^2-8AB+A\beta+2A\delta-B\alpha+2B\delta+8D-2\alpha+2\beta.
\end{align*}
Therefore the overlap ambiguity $BDA$ is resolvable and so every ambiguity is resolvable.  The theorem follows from these comments along with the Diamond Lemma \cite[p.181]{bergman}.
\end{proof}

\begin{lem}\label{lem:Cebasis}
The elements 
\begin{gather}\label{e:Cebasis}
\alpha^r\delta^s\beta^t 
\qquad 
\hbox{for all $r,s,t\in \N$}
\end{gather}
form an $\F$-basis for $\Ce$.
\end{lem}
\begin{proof}
Recall that $\Ce$ is defined to be the commutative $\F$-subalgebra of $\Re$ generated by $\alpha$, $\beta$, $\gamma$, $\delta$. 
By Lemma \ref{lem:alp&bet&gam&del}(iii), $\Ce$ is generated by $\alpha$, $\beta$, $\delta$. Therefore $\Ce$ is spanned by (\ref{e:Cebasis}) as an $\F$-vector space.  By Theorem  \ref{thm:basisURA}, these elements (\ref{e:Cebasis}) are linearly independent over $\F$.
The result follows. 
\end{proof}

\begin{defn}\label{def:Rn}
For each $n\in \N$, define $\Re_n$ to be the $\F$-subspace of $\Re$ spanned by
\begin{gather*}
A^i D^j B^k \alpha^r \delta^s \beta^t
\qquad \quad
\hbox{for all $i,j,k,r,s,t\in \N$ with $i+2 j+k+r+s+t\leq n$}.
\end{gather*}
Note that $\Re_0=\F 1$. For notational convenience, we define $\Re_{-1}$ to be the zero subspace of $\Re$.
\end{defn}

The following result is an immediate consequence of Theorem \ref{thm:basisURA} and Definition \ref{def:Rn}.
\begin{lem}\label{lem:R1basis}
The elements
\begin{gather*}
A,
\quad 
B,
\quad 
\alpha,
\quad
\beta,
\quad 
\delta,
\quad
1
\end{gather*}
form an $\F$-basis for the vector space $\Re_1$.
\end{lem}

\begin{lem}\label{lem:C=0}
Each of 
\begin{gather*}
A,
\quad 
B,
\quad 
C,
\quad 
\alpha,
\quad
\beta,
\quad 
\gamma,
\quad 
\delta,
\quad
1
\end{gather*}
is in $\Re_1$.
\end{lem}
\begin{proof}
Use (\ref{r:delta}), Lemma \ref{lem:alp&bet&gam&del}(iii), and Lemma \ref{lem:R1basis}. 
\end{proof}

\begin{lem}\label{lem:D=0}
The following equations hold in $\Re$:
\begin{enumerate}

\item $D=0 \pmod{\Re_1^2}$.

\item 
$DA=AD \pmod{\Re_1^2}$.

\item 
$DB =BD \pmod{\Re_1^2}$.
\end{enumerate}
\end{lem}
\begin{proof}
Use  Lemma \ref{lem:R1basis} to evaluate each of (\ref{r:D})--(\ref{r:beta}) modulo $\Re_1^2$.
\end{proof}

\begin{lem}\label{lem:ADBA}
For all $i,j,k\in \N$, the following equations hold in $\Re$:
\begin{gather*}
A^i D^j B^k A
=\left\{
\begin{array}{ll}
A^{i+1} D^{j} 
\pmod{\Re_1^{i+2j}} 
\qquad 
&\hbox{if $k=0$},
\\
A^{i+1} D^{j} B^k
-2k  A^{i} D^{j+1} B^{k-1}
\pmod{\Re_1^{i+2j+k}}
\qquad 
&\hbox{if $k\not=0$}.
\end{array}
\right.
\end{gather*}
\end{lem}
\begin{proof}
We proceed by induction on $k$.  The case when $k=0$ is immediate from Lemma \ref{lem:D=0}(ii). 
Suppose that $k\geq 1$.
By (\ref{r:D}), the term $A^i D^j B^{k} A$ is equal to
\begin{gather*}
A^i D^j B^{k-1} A B-2A^i D^j B^{k-1} D.
\end{gather*}
By the inductive hypothesis, we have 
$$
A^i D^j B^{k-1} A B=A^{i+1} D^j B^{k}-2(k-1) A^{i} D^{j+1} B^{k-1}
\pmod{\Re_1^{i+2j+k}}.
$$
By Lemma \ref{lem:D=0}(iii), we have 
$$A^i D^j B^{k-1} D=
A^i D^{j+1} B^{k-1} 
\pmod{\Re_1^{i+2j+k}}.
$$  
Combining the above comments, 
we see that 
$$A^i D^j B^{k} A=A^{i+1} D^j B^{k}-2k A^{i} D^{j+1} B^{k-1}
\pmod{\Re_1^{i+2j+k}}.$$
The result follows.
\end{proof}

\begin{lem}\label{lem:Rn=R1^n}
We have $\Re_n=\Re_1^n$ for all $n\in \N$.  
\end{lem}
\begin{proof}
Fix $n\in \N$. Let $i,j,k,r,s,t\in \N$ with $i+2j+k+r+s+t\leq n$. It follows from Lemma \ref{lem:C=0} and Lemma \ref{lem:D=0}(i) that 
$$
A^i D^j B^k \alpha^r \delta^s \beta^t \in \Re_1^n.
$$
By Definition \ref{def:Rn} this shows that $\Re_n\subseteq \Re_1^n$.

To obtain the reverse inclusion, we proceed by induction on $n$. The case when $n=0$ is trivial. 
 Suppose that $n\geq 1$.
Let $i,j,k,r,s,t\in \N$ with $i+2 j+k+r+s+t\leq n-1$. 
We show that $A^i D^j B^k \alpha^r \delta^s \beta^t X\in \Re_n$ for each basis element $X$ from Lemma \ref{lem:R1basis}.
By Lemma \ref{lem:alp&bet&gam&del}(ii), this statement is true for each $X\in \{B,\alpha,\delta,\beta, 1\}$. 
By Lemma \ref{lem:ADBA} and since $\Re_{n-1}=\Re_1^{n-1}$ by the inductive hypothesis, the monomial $A^i D^j B^k \alpha^r \delta^s \beta^t A$ is equal to 
\begin{gather}\label{e:multiplyA}
\left\{
\begin{array}{ll}
A^{i+1} D^j \alpha^r \delta^s \beta^t
\pmod{\Re_{n-1}}
\qquad &\hbox{if $k=0$},
\\
(A^{i+1} D^j B^k
-2k A^i D^{j+1} B^{k-1}) \alpha^r \delta^s \beta^t
\pmod{\Re_{n-1}}
\qquad 
&\hbox{if $k\not=0$}.
\end{array}
\right.
\end{gather}
By Definition \ref{def:Rn}, each term in (\ref{e:multiplyA}) lies in $\Re_n$. It follows from these comments that $A^i D^j B^k \alpha^r \delta^s \beta^t A\in \Re_n$. Thus, we have shown that 
\begin{gather}\label{e:Rn-1R}
\Re_{n-1}\cdot \Re_1\subseteq \Re_n.
\end{gather}
By the inductive hypothesis, the left-hand side of (\ref{e:Rn-1R}) is equal to $\Re_1^n$. Thus $\Re_1^n\subseteq \Re_n$ and so $\Re_1^n=\Re_n$.
\end{proof}

\begin{thm}\label{thm:filtration}
The $\F$-subspaces $\{\Re_n\}_{n\in \N}$ of $\Re$ satisfy the following properties:
\begin{enumerate}
\item $\Re=\bigcup_{n\in \N} \Re_n$,

\item $\Re_n\subseteq \Re_{n+1}$ for all $n\in \N$,

\item $\Re_m\cdot \Re_n=\Re_{m+n}$ for all $m,n\in \N$.
\end{enumerate}
Moreover, the $\F$-subspaces $\{\Re_n\}_{n\in \N}$ form an $\N$-filtration of $\Re$.
\end{thm}
\begin{proof} Assertions (i) and (ii) are immediate from Theorem \ref{thm:basisURA} and Definition \ref{def:Rn}, respectively. By Lemma \ref{lem:Rn=R1^n}, both $\Re_m\cdot \Re_n$ and $\Re_{m+n}$ are equal to $\Re_1^{m+n}$. Assertion (iii) follows.
\end{proof}

We now give two additional results concerning the $\N$-filtration of $\Re$ from Theorem \ref{thm:filtration}.

\begin{lem}\label{lem:Rn&D6}
With reference to Proposition \ref{prop:D6},  
 $\Re_n$ is invariant under the $D_6$-action on $\Re$ for each $n\in \N$.
\end{lem}
\begin{proof}
By Proposition \ref{prop:D6}, $\Re_1$ is invariant under the $D_6$-action on $\Re$.  The result follows from this along with Lemma  \ref{lem:Rn=R1^n}.
\end{proof}

\begin{lem}\label{lem:Rn/Rn-1}
For each $n\in\N$, the cosets 
\begin{gather}\label{e:Rn/Rn-1}
A^i D^j B^k \alpha^r \delta^s \beta^t
+\Re_{n-1}
\qquad 
\hbox{for all $i,j,k,\ell,r,s\in \N$ with $i+2 j+k+r+s+t=n$}
\end{gather}
form an $\F$-basis for the $\F$-vector space $\Re_n/\Re_{n-1}$.
\end{lem}
\begin{proof}
By Definition \ref{def:Rn}, $\Re_n/\Re_{n-1}$ is spanned by the elements in  (\ref{e:Rn/Rn-1}). The $\F$-linear independence of the elements in (\ref{e:Rn/Rn-1}) is immediate from Theorem \ref{thm:basisURA}. The result follows.
\end{proof}

We conclude this section with an application of the above results. In \cite{SH:2017-2}, we will provide another proof of the following proposition.

\begin{prop}\label{prop:no0divisor}
The $\F$-algebra $\Re$ has no zero divisors.
\end{prop}
\begin{proof}
Let $u,v$ denote two nonzero elements of $\Re$.  We show that the product $uv$ is nonzero.  
By Theorem \ref{thm:filtration}(i),(ii) there exists a unique $m\in \N$ such that $u\in \Re_m$ but $u\notin \Re_{m-1}$. 
Similarly, there exists a unique $n\in \N$ such that $v\in \Re_n$ but $v\notin \Re_{n-1}$. 
Let $M$ (resp. $N$) denote the set of all $(i,j,k,r,s,t)\in \N^6$ with $i+2j+k+r+s+t$ equal to $m$ (resp. $n$).  
We now consider the cosets $u+\Re_{m-1}$ and $v+\Re_{n-1}$.  
By Lemma \ref{lem:Rn/Rn-1} and the above comments, the coset $u+\Re_{m-1}$ (resp. $v+\Re_{m-1}$) can be uniquely written as an 
$\F$-linear combination of $A^i D^{j} B^{k} \alpha^{r} \delta^s \beta^t+\Re_{m-1}$ (resp. $A^i D^{j} B^{k} \alpha^{r} \delta^s \beta^t+\Re_{n-1}$) for $(i,j,k,r,s,t)$ in $M$ (resp. $N$),
with at least one nonzero coefficient.  Let $(i_u,j_u,k_u,r_u,s_u,t_u)$ (resp. $(i_v,j_v,k_v,r_v,s_v,t_v)$) denote the greatest element of $M$ (resp. $N$), with respect to the lexicographical ordering, that has nonzero coefficient in the aforementioned linear combination. 
We refer to the coset 
$A^{i_u} D^{j_u} B^{k_u} \alpha^{r_u} \delta^{s_u} \beta^{t_u}+\Re_{m-1}$
(resp. $A^{i_v} D^{j_v} B^{k_v} \alpha^{r_v} \delta^{s_v} \beta^{t_v}+\Re_{n-1}$)
as the \emph{leading term} of $u$ (resp. $v$). 

By  Lemma \ref{lem:D=0}(iii) and Lemma \ref{lem:ADBA}, the product of 
$A^i D^j B^k \alpha^r \delta^s \beta^t$ and 
$A^{i'} D^{j'} B^{k'}  \alpha^{r'} \delta^{s'} \beta^{t'}$ has the leading term
$$
A^{i+i'} D^{j+j'} B^{k+k'} 
\alpha^{r+r'}\delta^{s+s'} \beta^{t+t'}+\Re_{m+n-1}
$$
for all $(i,j,k,\ell,r,s)\in M$ and all $(i',j',k',\ell',r',s')\in N$. 
By the above comment and since the lexicographic order is compatible with addition, the product $uv$ has the leading term 
$$
A^{i_u+i_v} D^{j_u+j_v} B^{k_u+k_v} 
\alpha^{r_u+r_v} \delta^{s_u+s_v} \beta^{t_u+t_v}+\Re_{m+n-1}.
$$ In particular, $u v$ is nonzero. The result follows.
\end{proof}

\section{The $D_6$-symmetric Casimir elements of $\Re$}

In Section \ref{section:D6}, we gave a faithful $D_6$-action on the Racah algebra $\Re$.
In this section, we define the Casimir class of $\Re$ and show that it is contained in the center of $\Re$.  We also introduce a set of three Casimir elements which are invariant under the $D_6$-action on $\Re$.

Let $Z(\Re)$ denote the center of $\Re$.  
Let $\Ce$ denote the $\F$-subalgebra of $\Re$ generated by $\alpha$, $\beta$, $\gamma$, $\delta$. 
By Lemma \ref{lem:alp&bet&gam&del}(ii), $\Ce$ is contained in $Z(\Re)$.

\begin{lem}\label{lemma:D6&Ce}
With reference to Proposition \ref{prop:D6}, 
$\Ce$ is invariant under the $D_6$-action of $\Re$. 
\end{lem}
\begin{proof}
Use Proposition \ref{prop:D6}.
\end{proof}

In \cite[Section 2]{integrable2014-2}, the authors introduced a certain Casimir element of $\Re$.   
Inspired by the expression \cite[Equation (2.5)]{integrable2014-2}, we introduce the following definition.

\begin{defn}\label{defn:CasClass}
The coset 
\begin{gather*}
D^2+A^2+B^2
+\frac{(\delta+2)\{A,B\}-\{A^2,B\}-\{A,B^2\}}{2}
+A (\beta-\delta)
+B (\delta-\alpha)+\Ce
\end{gather*}
is called the {\it Casimir class} of $\Re$. Each element of the Casimir class of $\Re$ is called a {\it Casimir element} of $\Re$. 
\end{defn}

Define the three elements $\Omega_A$, $\Omega_B$, $\Omega_C$ of $\Re$ by
\begin{eqnarray}
\Omega_A
&=&
D^2
+
\frac{B A C
+C A B}{2}
+ A^2
+B \gamma
-C \beta
-A \delta,
\label{eq:CasA}
\\
\Omega_B
&=&
D^2
+
\frac{C B A
+A B C}{2}
+ B^2
+C \alpha 
-A \gamma
-B\delta,
\label{eq:CasB}
\\
\Omega_C
&=&
D^2
+
\frac{A C B
+B C A}{2}
+ C^2 
+A \beta 
-B\alpha
-C\delta.
\label{eq:CasC}
\end{eqnarray}

\begin{lem}\label{lem:D6&OmegaABC}
With reference to Proposition \ref{prop:D6}, the set $\{\Omega_A,\Omega_B,\Omega_C\}$ is invariant under the $D_6$-action on $\Re$.
The actions of $\sigma$ and $\tau$ on  $\Omega_A,\Omega_B,\Omega_C$ are as follows:
\begin{table}[H]
\centering
\extrarowheight=3pt
\begin{tabular}{c|ccc}
$u$  &$\Omega_A$ &$\Omega_B$ &$\Omega_C$ 
\\

\midrule[1pt]

$\sigma(u)$ &$\Omega_B$ &$\Omega_A$ &$\Omega_C$
\\
$\tau(u)$ &$\Omega_B$ &$\Omega_C$ &$\Omega_A$ 
\end{tabular}
\end{table}
\end{lem}
\begin{proof}
Use Proposition \ref{prop:D6} along with (\ref{eq:CasA})--(\ref{eq:CasC}).
\end{proof}

\begin{prop}\label{prop:CasABACas}
Each of $\Omega_A$, $\Omega_B$, $\Omega_C$ is a Casimir element of $\Re$.
\end{prop}
\begin{proof}
We first show that $\Omega_C$ is a Casimir element of $\Re$.  By \eqref{r:delta}, $C=\delta-A-B$.  In the expression for $\Omega_C$ given in \eqref{eq:CasC}, eliminate $C$ using this result.  Rearrange the terms to obtain 
$$
\Omega_C=D^2+A^2+B^2
+\frac{(\delta+2)\{A,B\}-\{A^2,B\}-\{A,B^2\}}{2}
+A (\beta-\delta)
+B (\delta-\alpha).
$$
Hence $\Omega_C$ is a Casimir element of $\Re$ by Definition \ref{defn:CasClass}.

We next show that $\Omega_B$ is a Casimir element of $\Re$.  Consider the difference $\Omega_B-\Omega_C$.  
By (\ref{eq:CasB}) and (\ref{eq:CasC}),  $\Omega_B-\Omega_C$ equal to the sum of the following three terms:
\begin{gather}
\frac{CBA-BCA+ABC-ACB}{2},\label{eq:sum1}\\
B^2-C^2-B\delta +C\delta,\label{eq:sum2}\\
\alpha(B+C)-A(\beta+\gamma).\label{eq:sum3}
\end{gather}
We consider each of these summands individually.  
Observe that the expression \eqref{eq:sum1} can be written as $\frac{1}{2}[A,[B,C]]$. 
Using (\ref{r:D}) and (\ref{r:alpha}) to simplify this expression, we find that \eqref{eq:sum1} is equal to 
\begin{gather}\label{eq:part1}
BA-AC+\alpha.
\end{gather}
We now consider \eqref{eq:sum2}.  
Using \eqref{r:delta} to eliminate $\delta$ in \eqref{eq:sum2} and \eqref{r:D} to simplify the result, we find that \eqref{eq:sum2} is equal to 
\begin{gather}\label{eq:part2}
-2D-BA+CA.
\end{gather}
We now consider \eqref{eq:sum3}.   Recall from Lemma \ref{lem:alp&bet&gam&del}(iii) that $\beta+\gamma=-\alpha$.  It follows from this that \eqref{eq:sum3} is equal to $\alpha(A+B+C)$. By this and (\ref{r:delta}), the expression \eqref{eq:sum3} is equal to
\begin{gather}\label{eq:part3}
\alpha\delta.
\end{gather}
From the above comments we see that \eqref{eq:sum1}, \eqref{eq:sum2}, \eqref{eq:sum3} is equal to  \eqref{eq:part1}, \eqref{eq:part2}, \eqref{eq:part3} respectively.  
It follows from this that $\Omega_B-\Omega_C=
[C,A]-2D+\alpha+\alpha\delta$.
The right-hand side of this equation can be simplified using \eqref{r:D} to obtain 
$\Omega_B-\Omega_C=\alpha+\alpha\delta$.  So $\Omega_B,\Omega_C$ differ by an element of $\Ce$ and thus 
\begin{gather}\label{eq:OmegaAB}
\Omega_B=\Omega_C
\pmod{\Ce}.
\end{gather}
Hence $\Omega_B$ is a Casimir element of $\Re$ by Definition \ref{defn:CasClass}. 

We now show that $\Omega_A$ is a Casimir element of $\Re$.  Recall from Lemma \ref{lemma:D6&Ce} that $\Ce$ is invariant under $\tau$.  Recall from Lemma \ref{lem:D6&OmegaABC} that $\tau$ sends $\Omega_B$, $\Omega_C$ to $\Omega_C$, $\Omega_A$ respectively. Thus,
applying $\tau$ to (\ref{eq:OmegaAB}) yields 
$$
\Omega_C\equiv\Omega_A
\pmod{\Ce}.
$$ 
Hence $\Omega_A$ is a Casimir element of $\Re$. 
\end{proof}

\begin{cor}
The $D_6$-symmetric Casimir elements $\Omega_A$, $\Omega_B$, $\Omega_C$ of $\Re$ are mutually distinct.
\end{cor}
\begin{proof}
Recall from the proof of Proposition \ref{prop:CasABACas} that 
\begin{gather}\label{OB-OC}
\Omega_B-\Omega_C=\alpha+\alpha\delta.
\end{gather}
By Lemma \ref{lem:Cebasis} the right-hand side of (\ref{OB-OC}) is nonzero. Hence $\Omega_B\not=\Omega_C$. Applying $\sigma$ to (\ref{OB-OC}) yields that 
\begin{gather}\label{OA-OC}
\Omega_A-\Omega_C=\beta+\delta\beta
\end{gather}
is nonzero by Lemma \ref{lem:Cebasis} again. Hence $\Omega_A\not=\Omega_C$. Applying $\tau$ to (\ref{OA-OC}) yields that 
\begin{gather}\label{OB-OA}
\Omega_B-\Omega_A=-\gamma-\gamma\delta.
\end{gather}
By Lemma \ref{lem:alp&bet&gam&del} the right-hand side of (\ref{OB-OA}) is equal to $\alpha+\beta+\alpha\delta+\delta\beta$. It follows from Lemma \ref{lem:Cebasis} that $\Omega_A\not=\Omega_B$. The corollary follows.
\end{proof}

In light of Lemma \ref{lem:D6&OmegaABC} and Proposition \ref{prop:CasABACas}, we call $\Omega_A$, $\Omega_B$, $\Omega_C$ the {\it $D_6$-symmetric Casimir elements of $\Re$}.

\begin{lem}\label{lemma:OmegaABC&D6}
With reference to Proposition \ref{prop:D6}, the Casimir class of $\Re$ is invariant under the $D_6$-action on $\Re$. 
\end{lem}
\begin{proof} 
By Proposition \ref{prop:CasABACas}, the Casimir class is equal to 
\begin{gather}\label{e:Casclass}
\Omega_A+\Ce=\Omega_B+\Ce=\Omega_C+\Ce.
\end{gather}
By Lemma \ref{lemma:D6&Ce} and Lemma \ref{lem:D6&OmegaABC}, the coset in (\ref{e:Casclass}) is invariant under the $D_6$-action on $\Re$.  The result follows.
\end{proof}

\begin{prop}\label{prop:CasCentral}
Each Casimir element of $\Re$ is central in $\Re$.
\end{prop}
\begin{proof}
We first show that $\Omega_A$ commutes with $A$.  
Using (\ref{eq:CasA}), a direct calculation yields
\begin{gather}\label{eq:OmegaAA}
[\Omega_A,A]=
[D^2,A]
+\frac{[BAC,A]+[CAB,A]}{2}
+[B,A]\gamma-[C,A]\beta.
\end{gather}
We now simplify the terms on the right-hand side of \eqref{eq:OmegaAA}. 
By (\ref{r:alpha}), 
\begin{align*}
D^2A&=DAC+DAD-DBA-D\alpha, 
\\
AD^2&=BAD+DAD-ACD+D\alpha.
\end{align*}
Thus
\begin{align*}
[D^2,A]=ACD+DAC-BAD-DBA-2D\alpha. 
\end{align*}
It follows from (\ref{r:D})  that 
\begin{align*}
BACA
&=
BA^2 C+2BAD,
\\
ABAC
&=
BA^2 C+2DAC.
\end{align*}
Thus
\begin{gather}\label{e:[BAC,A]}
[BAC,A]=2(BAD-DAC).
\end{gather}
By Proposition \ref{prop:D6}, the action of $\sigma\cdot \tau$ on $\Re$ is as the  $\F$-algebra automorphism of $\Re$ that fixes $A$ and sends 
$$
B\mapsto C,
\qquad 
C\mapsto B,
\qquad 
D\mapsto -D.
$$
Thus, applying $\sigma\cdot \tau$ to (\ref{e:[BAC,A]}) we obtain 
$$
[CAB,A]=2(DAB-CAD).
$$
By (\ref{r:D}), we have $[B,A]\gamma=-2D\gamma $ and $[C,A]\beta=2 D\beta$. 

Combining \eqref{eq:OmegaAA} with the above results concerning the terms on the right-hand side of \eqref{eq:OmegaAA}, we obtain 
\begin{gather}
[\Omega_A, A]=[A,C]D-D[B,A]-2(\alpha+\beta+\gamma) D.\label{eq:OmegaAA2}
\end{gather}
Observe that the right-hand side of \eqref{eq:OmegaAA2} is zero by \eqref{r:D} and Lemma \ref{lem:alp&bet&gam&del}(iii).  
Hence
\begin{gather}\label{e:OAA=0}
[\Omega_A, A]=0.
\end{gather}
By Proposition \ref{prop:D6}  and Lemma \ref{lem:D6&OmegaABC}, applying $\tau$,  $\tau^2$ to (\ref{e:OAA=0}) yields that  
\begin{gather}\label{e:OBB=0}
[\Omega_B,B]=0, 
\qquad 
[\Omega_C,C]=0 
\end{gather}
respectively.
The elements $\Omega_A$, $\Omega_B$, $\Omega_C$ are Casimir elements of $\Re$ by Proposition \ref{prop:CasABACas}. 
Since any two Casimir elements differ by a central element of $\Re$, 
 (\ref{e:OAA=0}) and (\ref{e:OBB=0}) imply that every Casimir element of $\Re$ commutes with $A$, $B$, $C$. 
 Recall from Lemma \ref{lem:alp&bet&gam&del}(i) that $\Re$ is generated by $A$, $B$, $C$. 
The proposition follows.
\end{proof}

\section{The center of $\Re$ at characteristic zero}

In this section, we show that any Casimir element $\Omega$ is algebraically independent over $\Ce$. Moreover, we show that the center $Z(\Re)$ of $\Re$ is equal to $\Ce[\Omega]$ provided that ${\rm char}\,\F=0$. Many of the results in this section involve Casimir elements of $\Re$ which are congruent to $D^2$ modulo $R_3$.  The following lemma shows the existence of such elements.

\begin{lem}\label{lemma:Casmod}
For $\Omega\in \{\Omega_A,\Omega_B,\Omega_C\}$, we have $\Omega=D^2\bmod{\Re_3}.$
\end{lem}
\begin{proof}
The lemma is immediate from the forms \eqref{eq:CasA}--\eqref{eq:CasC} of $\Omega_A$, $\Omega_B$, $\Omega_C$ along with Lemma \ref{lem:C=0} and Theorem \ref{thm:filtration}(iii).
\end{proof}

\begin{lem}\label{lem:Omega=D^2}
Let $\Omega$ denote a Casimir element of $\Re$ with 
$\Omega=D^2 \bmod{\Re_3}$. Then 
$$
\Omega^n=D^{2n} \pmod{\Re_{4n-1}}
$$
for all $n\in \N$.
\end{lem}
\begin{proof}
We proceed by induction on $n$. The case when $n=0$ is trivial.   
Suppose that $n\geq 1$.  
Observe that 
$\Omega^n-D^{2n}$ 
is equal to 
\begin{gather}\label{eq:OmegaD}
\Omega(\Omega^{n-1}-D^{2n-2})+(\Omega-D^2) D^{2n-2}.
\end{gather}
We show that each summand in \eqref{eq:OmegaD} is in $\Re_{4n-1}$.  
By Lemma \ref{lem:D=0}(i) and Lemma \ref{lem:Rn=R1^n}, we know $D^2\in\Re_4$.  Since $\Omega=D^2 \bmod{\Re_3}$, it follows that $\Omega\in\Re_4$.  It then follows from Theorem \ref{thm:filtration}(iii) and the inductive hypothesis that the summand on the left in \eqref{eq:OmegaD} is in $\Re_{4n-1}$. 

By Lemma \ref{lem:D=0}(i) and Lemma \ref{lem:Rn=R1^n}, we know $D^{2n-2}\in\Re_{4n-4}$.  Since $\Omega=D^2 \bmod{\Re_3}$, it follows that $\Omega-D\in\Re_3$. It follows from these facts along with Theorem \ref{thm:filtration}(iii) that $(\Omega-D^2) D^{2n-2}\in \Re_{4n-1}$. 
The result follows.
\end{proof}

\begin{lem}\label{lem:3basisRn/Rn-1}
Let $\Omega$ denote a Casimir element of $\Re$ with 
$\Omega=D^2 \bmod{\Re_3}$. 
For each $n\in \N$, 
 the cosets
\begin{gather*}
A^i D^j B^k
\Omega^\ell
\alpha^r
\delta^s
\beta^t+\Re_{n-1}
\qquad \quad
\begin{split}
&\hbox{for all $i,k,\ell,r,s,t\in \N$ and $j\in \{0,1\}$}
\\
&\hbox{with $i+2j+k+4\ell+r+s=n$}
\end{split}
\end{gather*}
form an $\F$-basis for $\Re_n/\Re_{n-1}$.
\end{lem}
\begin{proof}
Let $I$ denote the set consisting of all $(i,j,k,r,s,t)\in \N^6$ with 
$$
i+2j+k+r+s+t=n.
$$ 
Let $J$ denote the set consisting of all $(i,j,k,\ell,r,s,t)\in \N^7$ with 
$$
j\in\{0,1\}
\quad 
\hbox{and} 
\quad 
i+2j+k+4\ell+r+s+t=n.
$$ 
By Proposition \ref{prop:CasCentral} and Lemma \ref{lem:Omega=D^2},  
$$
A^i D^j B^k \alpha^r \delta^s \beta^t
=
A^i D^{j-2\lfloor \frac{j}{2}\rfloor} B^k \Omega^{\lfloor \frac{j}{2}\rfloor}
\alpha^r \delta^s \beta^t
\pmod{\Re_{n-1}}
$$
for all $(i,j,k,r,s,t)\in I$. The map $I\to J$ given by 
\begin{eqnarray*}
(i,j,k,r,s,t) 
&\mapsto &
(i,j-2\lfloor j/2\rfloor,k,\lfloor j/2\rfloor,r,s,t)
\qquad 
\hbox{for all $(i,j,k,r,s,t)\in I$}
\end{eqnarray*}
is a bijection.  Its inverse is the map $J\to I$ that sends 
\begin{eqnarray*}
(i,j,k,\ell, r,s,t) 
&\mapsto &
(i,j+2\ell,k,r,s,t)
\qquad 
\hbox{for all $(i,j,k,\ell,r,s,t)\in J$}.
\end{eqnarray*}
The result follows from these comments along with Lemma \ref{lem:Rn/Rn-1}.
\end{proof}

\begin{lem}\label{lem:3basisRn}
Let $\Omega$ denote a Casimir element of $\Re$ with 
$\Omega=D^2 \bmod{\Re_3}$. 
For each $n\in \N$, the elements
\begin{gather*}
A^i D^j B^k
\Omega^\ell
\alpha^r
\delta^s
\beta^t
\qquad \quad
\begin{split}
&\hbox{for all $i,k,\ell,r,s,t\in \N$ and $j\in \{0,1\}$}
\\
&\hbox{with $i+2j+k+4\ell+r+s\leq n$}
\end{split}
\end{gather*}
form an $\F$-basis for $\Re_n$.
\end{lem}
\begin{proof}
The result follows from Lemma \ref{lem:3basisRn/Rn-1} and a routine induction on $n$.
\end{proof}

Combining Theorem \ref{thm:filtration}(i) and Lemma \ref{lem:3basisRn} we immediately obtain the following result.

\begin{thm}\label{thm:basisURA2}
Let $\Omega$ denote a Casimir element of $\Re$ with 
$\Omega=D^2 \bmod{\Re_3}$. 
The elements
\begin{gather*}
A^i D^j B^k
\Omega^\ell
\alpha^r
\delta^s
\beta^t
\qquad \quad
\hbox{for all $i,k,\ell,r,s,t\in \N$ and $j\in \{0,1\}$}
\end{gather*}
form an $\F$-basis for $\Re$.
\end{thm}

\begin{cor}\label{cor:linindep_Cas&abc}
Let $\Omega$ denote a Casimir element of $\Re$ with 
$\Omega=D^2 \bmod{\Re_3}$. 
Then the elements
\begin{gather*}
\Omega^\ell
\alpha^r
\delta^s
\beta^t
\qquad \quad
\hbox{for all $\ell,r,s,t\in \N$}
\end{gather*}
are linearly independent over $\F$.
\end{cor}

\begin{lem}\label{lem:Cas_algind}
Each Casimir element of $\Re$ is algebraically independent over $\Ce$.
\end{lem}
\begin{proof}
Let $\Omega$ denote a Casimir element of $\Re$ with $\Omega=D^2\bmod{\Re_3}$. 
By Lemma \ref{lem:Cebasis} and Corollary \ref{cor:linindep_Cas&abc}, the corollary holds for $\Omega$. By Definition \ref{defn:CasClass}, any two  Casimir elements of $\Re$ differ by an element of $\Ce$. The result follows from these comments.
\end{proof}

\begin{lem}\label{lem:[A,ADB]}
For all $i,j,k\in \N$, the following equations hold in $\Re$:
\begin{enumerate}
\item
$
[A,A^i D^j B^k] =
\left\{ 
\begin{array}{ll}
0
\pmod{\Re_{i+2j+k}}
\qquad &\hbox{if $k=0$},
\\
2k  A^{i} D^{j+1} B^{k-1}
\pmod{\Re_{i+2j+k}}
\qquad &\hbox{if $k\not=0$}.
\end{array}
\right.
$

\item
$
[B,A^i D^j B^k] =
\left\{ 
\begin{array}{ll}
0
\pmod{\Re_{i+2j+k}}
\qquad &\hbox{if $i=0$},
\\
-2i  A^{i-1} D^{j+1} B^{k}
\pmod{\Re_{i+2j+k}}
\qquad &\hbox{if $i\not=0$}.
\end{array}
\right.
$
\end{enumerate}
\end{lem}
\begin{proof}
Assertion (i) follows from Lemma \ref{lem:ADBA} and Lemma \ref{lem:Rn=R1^n}. 
Recall from Lemma \ref{lem:Rn&D6} that $\Re_{i+2j+k}$ is invariant under $\sigma$. We obtain (ii)  by applying $\sigma$ to (i).
\end{proof}

\begin{prop}\label{prop:Z(R)capRn}
Assume that ${\rm char}\,\F=0$. 
Let $\Omega$ denote a Casimir element of $\Re$ with 
$\Omega=D^2 \bmod{\Re_3}$. 
For each $n\in \N$, the elements 
\begin{gather}\label{e:basisZ&Rn}
\Omega^\ell \alpha^r \delta^s \beta^t
\qquad 
\hbox{for all $\ell,r,s,t\in \N$ with $4\ell+r+s+t\leq n$}
\end{gather}
form an $\F$-basis for $Z(\Re)\cap \Re_n$.
\end{prop}
\begin{proof}
By Lemma \ref{lem:alp&bet&gam&del}(ii) and Proposition \ref{prop:CasCentral}, the elements in (\ref{e:basisZ&Rn}) are central in $\Re$. By Lemma \ref{lem:3basisRn}, the elements in (\ref{e:basisZ&Rn}) are contained in $\Re_n$. Hence the elements in (\ref{e:basisZ&Rn}) are contained in $Z(\Re)\cap \Re_n$. 

The elements in \eqref{e:basisZ&Rn} are linearly independent over $\F$ by Corollary \ref{cor:linindep_Cas&abc}.  All that remains to show is that the elements in \eqref{e:basisZ&Rn} span  $Z(\Re)\cap \Re_n$.  We proceed by induction on $n$.  The case when $n=0$ is trivial. 
Assume that $n\geq 1$. Let $R\in Z(\Re)\cap \Re_n$ be given. For $0\leq m\leq n$, let $I(m)$ denote the set of all 
$(i,j,k,\ell,r,s,t)\in \N^7$ with 
$$
j\in \{0,1\} 
\quad \hbox{and}
\quad 
i+2j+k+4\ell+r+s+t=m.
$$  
For $(i,j,k,\ell,r,s,t)\in I(m)$, we let $c(i,j,k,\ell,r,s,t)$ denote the coefficient of $A^i D^j B^k \Omega^\ell \alpha^r \delta^s \beta^t$ in $R$ with respect to the  basis for $\Re_n$ given in Lemma \ref{lem:3basisRn}. Using this notation, we can write 
\begin{gather}\label{e:R}
R=\sum_{m=0}^n
\;
\sum_{(i,j,k,\ell,r,s,t)\in I(m)} 
c(i,j,k,\ell,r,s,t)
A^i D^j B^k \Omega^\ell \alpha^r \delta^s \beta^t.
\end{gather}
We will show that the coefficient $c(i,j,k,\ell,r,s,t)\in I(n)$ is equal to zero if $(i,j,k)\neq (0,0,0)$.  
Observe that  the commutator $[A,R]=0$ since $R\in Z(\Re)$. Consequently,  
\begin{gather}\label{e:[A,R]}
\sum_{(i,j,k,\ell,r,s,t)\in I(n)} 
c(i,j,k,\ell,r,s,t)
[A,A^i D^j B^k] \Omega^\ell \alpha^r \delta^s \beta^t=0
\pmod{\Re_{n}}.
\end{gather}
Using Lemma \ref{lem:[A,ADB]}(i) to simplify the left-hand side of (\ref{e:[A,R]}) yields 
\begin{gather}\label{e2:[A,R]}
\sum_{(i,j,k,\ell,r,s,t)\in I(n)\atop k\not=0}
2k\cdot c(i,j,k,\ell,r,s,t) A^{i}  D^{j+1} B^{k-1}
\Omega^\ell \alpha^r \delta^s \beta^t
=0
\pmod{\Re_{n}}.
\end{gather}
Since $j\in\{0,1\}$ and $\Omega=D^2\bmod{\Re_3}$, the equation (\ref{e2:[A,R]}) can be written as 
\begin{align*}
&\sum_{(i,0,k,\ell,r,s,t)\in I(n)\atop k\not=0}
2k\cdot c(i,0,k,\ell,r,s,t) A^{i}  D B^{k-1}
\Omega^\ell \alpha^r \delta^s \beta^t
\\
&\qquad +
\sum_{(i,1,k,\ell,r,s,t)\in I(n)\atop k\not=0}
2k\cdot c(i,1,k,\ell,r,s,t) A^{i}  B^{k-1}
\Omega^{\ell+1} \alpha^r \delta^s \beta^t=0 \pmod{\Re_n}.
\end{align*}
It follows from Lemma \ref{lem:3basisRn/Rn-1} and ${\rm char}\,\F=0$ that 
\begin{gather*}
c(i,j,k,\ell,r,s,t)=0 \qquad \hbox{for all $(i,j,k,\ell,r,s,t)\in I(n)$ with $k\not=0$}.
\end{gather*}
A similar argument concerning $[B,R]$ shows that  
\begin{gather*}
c(i,j,k,\ell,r,s,t)=0 \qquad \hbox{for all $(i,j,k,\ell,r,s,t)\in I(n)$ with $i\not=0$}.
\end{gather*}
It now remains to show that $(0,j,0,\ell,r,s,t)\in I(n)$ is equal to zero if $j\neq 0$.  
Let
\begin{gather}\label{e:S}
S=\sum_{(0,0,0,\ell,r,s,t)\in I(n)} 
c(0,0,0,\ell,r,s,t)
\Omega^\ell \alpha^r \delta^s \beta^t.
\end{gather}
By Proposition \ref{prop:CasCentral}, $S\in Z(\Re)$. 
By the above comments, we see that $R-S\in Z(\Re)$ and $R-S$ is equal to 
\begin{align}\label{e:R=S+..}
\sum_{(0,1,0,\ell,r,s,t)\in I(n)} 
c(0,1,0,\ell,r,s,t)
D \Omega^\ell \alpha^r \delta^s \beta^t
+
\sum_{m=0}^{n-1}\sum_{I(m)} 
c(i,j,k,\ell,r,s,t)
A^i D^j B^k \Omega^\ell \alpha^r \delta^s \beta^t.
\end{align}
Since $R-S\in Z(\Re)$, the commutator $[R-S,A]=0$. It follows that 
\begin{align}\label{e:R'1}
\sum_{(0,1,0,\ell,r,s,t)\in I(n)}
c(0,1,0,\ell,r,s,t)[D,A]\Omega^\ell \alpha^r \delta^s \beta^t 
\end{align}
and
\begin{align}\label{e:R'2}
\sum_{(i,j,k,\ell,r,s,t)\in I(n-1)}
c(i,j,k,\ell,r,s,t)
[A,A^i D^j B^k] \Omega^\ell \alpha^r \delta^s \beta^t.
\end{align}
are congruent modulo $\Re_{n-1}$.
By Proposition \ref{prop:presentation2}, the expression in (\ref{e:R'1}) is equal to 
\begin{gather}\label{e:leftR'}
\sum_{(0,1,0,\ell,r,s,t)\in I(n)}
c(0,1,0,\ell,r,s,t)(A\delta-A^2-2AB+2D)\Omega^\ell \alpha^r \delta^s \beta^t
\pmod{\Re_{n-1}}.
\end{gather}
By Lemma \ref{lem:[A,ADB]}(i), the expression in (\ref{e:R'2}) is equal to 
\begin{gather}\label{e:rightR'}
\sum_{(i,j,k,\ell,r,s,t)\in I(n-1)\atop k\not=0}
2k\cdot 
c(i,j,k,\ell,r,s,t)
A^i D^{j+1} B^{k-1}
\Omega^\ell \alpha^r \delta^s \beta^t
\pmod{\Re_{n-1}}.
\end{gather}
Let $\ell,r,s,t\in \N$ with $4\ell+r+s+t=n-2$ be given. Observe that the coefficient of 
$A^2 \Omega^\ell \alpha^r\delta^s\beta^t$
in (\ref{e:leftR'}) is 
$$
-c(0,1,0,\ell,r,s,t)
$$ 
and that the coefficient of 
$A^2 \Omega^\ell \alpha^r\delta^s\beta^t$ in (\ref{e:rightR'}) is zero. 
By Lemma \ref{lem:3basisRn/Rn-1}, 
this implies  that  
$$
c(0,1,0,\ell,r,s,t)=0
\qquad 
\hbox{for all $(0,1,0,\ell,r,s,t)\in I(n)$}.
$$
In other words, the sum on the left in (\ref{e:R=S+..}) is zero. Hence we obtain that 
$$
R-S\in  \Re_{n-1}. 
$$
By the inductive hypothesis, $R-S$ is an $\F$-linear combination of $\Omega^\ell \alpha^r \delta^s \beta^t$ for all $\ell,r,s,t\in \N$ with $4\ell+r+s+t\leq n-1$. Combined with (\ref{e:S}), these comments imply that  $R$ is an $\F$-linear combination of the elements in (\ref{e:basisZ&Rn}). The result follows.
\end{proof}

\begin{thm}\label{thm:Z(R)}
Assume that ${\rm char}\,\F=0$. 
Let $\Omega$ denote a Casimir element of $\Re$ with $\Omega=D^2 \bmod{\Re_3}$. Then the elements 
$$
\Omega^\ell \alpha^r \delta^s \beta^t
\qquad 
\hbox{for all $\ell,r,s,t\in \N$}
$$ 
form an $\F$-basis for $Z(\Re)$. 
\end{thm}
\begin{proof}
Use Theorem \ref{thm:filtration}(i) and Proposition \ref{prop:Z(R)capRn}.
\end{proof}

We conclude the present paper with two consequences of Theorem \ref{thm:Z(R)}.

\begin{lem}\label{lem:Z(R)=poly}
Assume that ${\rm char}\,\F=0$. Let $\Omega$ denote a Casimir element of $\Re$. Then $Z(\Re)=\Ce[\Omega]$.
\end{lem}
\begin{proof}
By Lemma \ref{lemma:Casmod}, $\Omega_A=D^2\bmod{\Re_3}$.  By Theorem \ref{thm:Z(R)}, we have $Z(\Re)=\Ce[\Omega_A]$.  Recall that any two Casimir elements differ by an element of $\Ce$.  The result follows from these comments.  
\end{proof}

\begin{lem}
Assume that ${\rm char}\,\F=0$. With reference to Proposition \ref{prop:D6},  $Z(\Re)$ is invariant under the $D_6$-action of $\Re$.
\end{lem}
\begin{proof}
Use Lemma \ref{lemma:D6&Ce}, Lemma \ref{lemma:OmegaABC&D6}, and Theorem  \ref{thm:Z(R)}.
\end{proof}

\bibliographystyle{amsplain}
\providecommand{\bysame}{\leavevmode\hbox to3em{\hrulefill}\thinspace}
\providecommand{\MR}{\relax\ifhmode\unskip\space\fi MR }
\providecommand{\MRhref}[2]{%
  \href{http://www.ams.org/mathscinet-getitem?mr=#1}{#2}
}
\providecommand{\href}[2]{#2}

\end{document}